\documentclass[11pt]{amsart}
\usepackage{amssymb}
\usepackage{amscd}
\usepackage{color}
\usepackage[usenames,dvipsnames]{xcolor}
\usepackage{url}
\usepackage{hyperref}

\numberwithin{equation}{section}

\theoremstyle{plain}
\newtheorem{theorem}[subsection]{Theorem}
\newtheorem{proposition}[subsection]{Proposition}
\newtheorem{lemma}[subsection]{Lemma}
\newtheorem{corollary}[subsection]{Corollary}

\theoremstyle{definition}

\newcommand{\Q}{\mathbb{Q}}
\newcommand{\Z}{\mathbb Z}

\newcommand{\R}{\mathbb{R}}

\begin{document}
%\reversemarginpar\marginpar{\vspace{-30pt}{\hfill{\small$\mathsf{\jobname}$}}}

\makeatletter
\def\author@andify{%
  \nxandlist {\unskip ,\penalty-1 \space\ignorespaces}%
    {\unskip {} \@@and~}%
    {\unskip \penalty-2 \space \@@and~}%
}
\makeatother

\title{Fourier series in $\text{BMO}$ with number theoretical implications}

\author{Fernando Chamizo}
\address{Departamento de Matem\'aticas and ICMAT \\ Universidad Aut\'onoma de Ma\-drid \\ Madrid 28049 \\ Spain}
\email{fernando.chamizo@uam.es}

\author{Antonio C\'ordoba}
\address{Departamento de Matem\'aticas and ICMAT \\ Universidad Aut\'onoma de Ma\-drid \\ Madrid 28049 \\ Spain}
\email{antonio.cordoba@uam.es}

\author{Adri\'an Ubis}
\address{Departamento de Matem\'aticas\\ Universidad Aut\'onoma de Madrid \\ Madrid 28049 \\ Spain}
\email{adrian.ubis@gmail.com}

\thanks{The authors are partially supported by the MTM2017-83496-P grant of the MICINN (Spain) and the first author and the second author are also supported by ``Severo Ochoa Programme for Centres of Excellence in R{\&}D'' (SEV-2015-0554)}



\begin{abstract}
We introduce an elementary argument to bound the $\textrm{BMO}$ seminorm of  Fourier series with gaps giving in particular
a sufficient condition for them to be in this space. Using finer techniques we carry out a detailed study of the series $\sum n^{-1}e^{2\pi i n^2 x}$
providing some insight into how much this $\text{BMO}$ Fourier series differs from defining an $L^\infty$ function. 
\end{abstract}

\maketitle

% \pagestyle{myheadings} %% running title
% \markboth{\hfill{\small$\mathsf{\jobname}$}}{{\small$\mathsf{\jobname}$}\hfill}

\section{Introduction}

An unpublished result of C.~Fefferman allows to characterize all possible Hardy inequalities for {functions in $H^1$ (analytic functions in the unit disc so that $\|f\|_{H^1}:=\sup_{r<1}\int_0^{2\pi}|f(re^{i\theta})|\; d\theta<\infty$)}. Namely (see \cite[Th.A]{dyakonov} \cite[p.264]{AnSh}), given a nonnegative sequence $\{a_k\}_{k=1}^\infty$
\begin{equation}\label{ghardy}
 \sum_{k=1}^\infty
 a_k|\widehat{f}(k)|
 \le
 C\| f\|_{H^1},
\end{equation}
where as usual 
$\widehat{f}(k)=\int_0^1e^{-2\pi ikx}f\big(e^{2\pi ix}\big)\; dx$,
holds for some constant $C$ and every $f\in H^1$ 
if and only if
\begin{equation}\label{f_cond}
 \sup_{N\in\Z^+}
 \sum_{j=1}^\infty
 \Big(
 \sum_{jN\le k<(j+1)N}
 a_k
 \Big)^2
 <\infty.
\end{equation}
Using the duality between $H^1$ and $\textrm{BMO}$ {(bounded mean oscillation), a result contained in the classical references \cite{fefferman} \cite{FeSt}, and a property of} $H^1$ multipliers \cite[Th.6.8]{duren} it turns out that \eqref{f_cond} is also a necessary and sufficient condition \cite[Cor.2]{SlSt} for 
\[
 \sum_{k=1}^\infty
 a_ke(kx)\in\textrm{BMO}
 \quad\text{when $a_k\ge 0$}
 \qquad\text{where }e(x):=e^{2\pi i x}.
\]
In fact if \eqref{f_cond} holds then any other Fourier series 
$\sum_{k=1}^\infty b_ke(nx)$ with $|b_n|\le a_n$ also belongs to $\textrm{BMO}$. 
Recall that $\textrm{BMO}$, or more precisely $\textrm{BMO}(\mathbb{T})$, is the space of $1$-periodic integrable functions $f$ such that
\[
 \|f\|_*
 :=
 \sup_{I\subset\mathbb{T}}\|f\|_I
 <\infty
 \qquad\text{with}\qquad
 \|f\|_I=|I|^{-1}\int_I|f-f_I|
\]
where $I$ is an interval of $\mathbb{T}$ and $f_I$ stands for the average of $f$ on $I$. Obviously $\textrm{BMO}\supset L^\infty$ and $\sum_{k=1}^\infty e(kx)/k$ proves that we have a proper inclusion.

Although \eqref{f_cond} provides a full characterization, some authors \cite{dyakonov} \cite{OsTo} have studied what kind of gaps in the nonzero values of $a_k$ can assure \eqref{ghardy} and other generalizations when these positive values are of certain type. 
For instance \cite[Cor.1]{dyakonov} states
\begin{equation}\label{hdyak}
 \sum_{k=1}^\infty
 \frac{|\widehat{f}(\nu_k)|}{k}
 \le
 C\| f\|_{H^1}
\end{equation}
whenever 
$\{\nu_k\}_{k=1}^\infty$ 
is   increasing and 
$\{\nu_k/k\}_{k=1}^\infty$
is nondecreasing. 
In fact the same can be proved under the weaker assumption 
$\inf k(\nu_{k+1}/\nu_k-1)>0$ \cite[Th.1]{dyakonov}. 
On the other hand, it is possible to find examples with  $k(\nu_{k+1}/\nu_k-1)$ going to zero at any rate {and} such that $\sum_{k=1}^\infty k^{-1}e(\nu_kx)$ violates \eqref{f_cond}{:} Hence it does not belong to $\textrm{BMO}$ and \eqref{hdyak} does not hold. 
Results of this kind appear in \cite{OsTo} in a broader scope. For instance, \cite[Th.4]{OsTo} for $q=1$, using duality gives that for $b_k\in\ell^2(\Z^+)$
\begin{equation}\label{osikiewicz}
 \sum_{k=1}^\infty
 b_k e(\nu_k x)\in\textrm{BMO}
 \qquad
 \text{if}\quad
 \frac{\nu_{k+1}}{\nu_k}\ge 1+\delta |b_k|
 \quad
 \text{for some $\delta>0$}.
\end{equation}
{As an aside, for  frequencies given by powers, replacing $\textrm{BMO}$ by $L^p$ leads to interesting open problems \cite{cordoba}.}

The proof of the characterization \eqref{f_cond} of all $\textrm{BMO}$ Fourier series with positive coefficients, as given in \cite{SlSt}, employs the $H^1$ and $\textrm{BMO}$ duality, the atomic decomposition due to Coifman \cite{coifman} and an elementary argument that was many years before successfully employed by Gallagher \cite{gallagher} to get a surprisingly simple proof of a large sieve  inequality in number theory (this remained probably unnoticed by harmonic analysts). 
The results in \cite{dyakonov} and \cite{OsTo} are obtained checking that some control on the gaps of the selected frequencies assures \eqref{f_cond}.

\

This paper has a double purpose. Firstly, having in mind the analytic ideas involved in the large sieve \cite{montgomery}, we derive in \S\ref{section_bounds}  a sufficient condition to have $f=\sum_{k=1}^\infty
 b_k e(\nu_k x)\in\textrm{BMO}$ and to estimate $\|f\|_*$ using 
Hilbert's inequality in the generalized form stated by Montgomery and Vaughan  \cite[(1.7)]{MoVa}
\begin{equation}\label{hilbert}
 \sum_{r\ne s}
 \frac{w_r\bar{w}_s}{\lambda_r-\lambda_s}
 \le
 \frac{3\pi}{2}
 \sum_r |w_r|^2\delta_r^{-1}
\end{equation}
where $\lambda_r$ are real numbers with $|\lambda_r-\lambda_s|\ge \delta_r>0$ for any $s\ne r$.
An advantage on this approach, is that the argument is short and completely elementary, because so it is the proof of \eqref{hilbert}, not depending on the duality result $(H^1)^*=\text{BMO}$. 

Secondly, in \S\ref{section_squares} we devote our efforts to show how other analytic number theory tools give precise information about the remarkable Fourier series
\[
 F(x)
 =
 \sum_{n=1}^\infty
 \frac{e(n^2 x)}n
\]
which is a critical case of some Fourier series considered by  several authors (e.g. \cite{jaffard}, \cite{SeUb}, \cite{ChUb}, {\cite{ChCo}}) and, following Weierstrass, related to Riemann's strategy in the search of continuous nowhere differentiable functions. Roughly speaking we are interested in how far is this $\text{BMO}$ function from being bounded. Namely we fully characterize  the points in which the series converges and we provide fine estimates for the measure of the level sets $\{x\in I\; :\; |F(x)-F_I|>\lambda\}$. In connection with this, recall that a seminal result by John and Nirenberg \cite{JoNi} asserts that
\begin{equation}\label{joni_ineq}
 \frac{1}{|I|}  {\big|\{x\in I\; :\; |f(x)-f_I|>\lambda\}\big|}
 \le C_1 \exp(-C_2 \lambda/\|f\|_*)
\end{equation}
for some constants $C_1$,  $C_2$ and any $f\in \text{BMO}$. It can be rephrased saying that the functions in $\text{BMO}\setminus L^\infty$ have at most logarithmic singularities.

\section{Bounds for the \texorpdfstring{$\text{BMO}$}{BMO} seminorm}\label{section_bounds}

To state our results we consider two sequences, representing frequencies and bounds for the Fourier coefficients
\begin{equation}\label{sequ}
 \{\nu_n\}_{n=1}^\infty\subset\Z^+
 \text{ increasing }
 \quad\text{ and }\quad
 \{a_n\}_{n=1}^\infty\in\ell^2(\Z^+)
 \text{ with }a_n>0.
\end{equation}
Associated to these sequences, we define for each $N\in\Z^+$
\begin{equation}\label{ST}
 S_N = \sum_{n=1}^{N}
 a_n\nu_n
 \qquad\text{and}\qquad
 T_N = \sum_{n=N+1}^{\infty}
 \frac{a_n^2}{M(n)}.
\end{equation}
where 
$M(n)=\min(\nu_n- \nu_{n-1}, \nu_{n+1}-\nu_n)$
with the special definition 
$M(N+1)=\nu_{N+2}-\nu_{N+1}$.

\begin{theorem}\label{general}
 With the notation {introduced before, let us} assume that
 \[
  \kappa:= 
  \sup_{0<\epsilon<1}
  \inf_{N\in\Z^+}
  \bigg(
  4\pi \epsilon S_N
  +\big(6\epsilon^{-1} T_N + 4 \sum_{n=N+1}^\infty a_n^2\big)^{1/2}
  \bigg)
  <\infty.
 \]
 Then
 \[
  \Big\|\sum_{n=1}^\infty
  b_ne(\nu_n x)\Big\|_*\le \kappa 
  \quad
  \text{for any $|b_n|\le a_n$}.
 \]
 In particular this Fourier series belongs to $\text{\rm BMO}$.
\end{theorem}

Choosing $N$ such that $S_N\le \epsilon^{-1}<S_{N+1}$, we deduce readily

\begin{corollary}\label{simpl}
 If $S_NT_N$ and $S_{N+1}/S_N$ are bounded then 
 \[
  \sum_{n=1}^\infty
  b_ne(\nu_n x)\in\textrm{\rm BMO}
  \quad
  \text{for any $|b_n|\le a_n$}
 \]
 and we have a Hardy type inequality
 \[
  \sum_{n=1}^\infty
  a_n\big|\widehat{f}(\nu_n)\big|
%  \lesssim
  \le C
  \|f\|_{H^1}
  \quad
  \text{for every $f\in H^1$}. 
 \]
\end{corollary}

The following result gives a general bound for $\|f\|_I$ on small intervals under a hypothesis slightly stronger than that of \eqref{osikiewicz}. Roughly speaking it asserts that with bigger gaps $\text{BMO}$ becomes closer to~$\text{VMO}$ {(vanishing mean oscillation)}.
\begin{proposition}\label{norm_I}
 Assume $\nu_{n+1}/\nu_n\ge 1+\delta \max(a_n, a_{n+1})$ for a certain $\delta>0$ and $n$ large enough. Then 
 \[
  \limsup_{|I|\to 0} 
  \Big\|\sum_{n=1}^\infty
  b_ne(\nu_n x)\Big\|_I
  \le 
  3(12\pi)^{1/3}\delta^{-1}
  \quad
  \text{for any $|b_n|\le a_n$}.
 \]
\end{proposition}

For instance, without appealing to \eqref{f_cond} and  without entering into the kind of Diophantine considerations appearing in the next section,  we have 
\[
 \sum_{n=1}^\infty
 \frac{e(n^kx)}{n}
 \in\text{BMO}
 \qquad\text{and}\qquad 
 \limsup_{|I|\to 0} 
  \Big\|
 \sum_{n=1}^\infty
 \frac{e(n^kx)}{n}
  \Big\|_I
  \le 
  \frac{3(12\pi)^{1/3}}{k}.
\]
The same bound applies if we introduce a arbitrarily chosen signs in the coefficients of this Fourier series. {Therefore} they do not affect significantly the mean oscillation{, a fact} which is far from being intuitive.

\section{A Fourier series in \texorpdfstring{$\text{BMO}$}{BMO}}\label{section_squares}

As mentioned in the introduction, in this section we perform a closer analysis of {the case}
\begin{equation}\label{riem1}
 F(x)
 =
 \sum_{n=1}^\infty
 \frac{e(n^2 x)}n.
\end{equation}
By our previous results we know that $F\in\text{BMO}$ and we {have sharp estimates for} $\|F\|_*$. In particular $F\in L^p$ for every $1\le p<\infty$. The celebrated theorem of Carleson  \cite{carleson} (see \cite{LaTh} for a simplification) implies that the series converges in the usual sense almost everywhere. On the other hand, for each irreducible fraction $p/q$ the normalized quadratic Gauss sum
\begin{equation}\label{g_norm}
 \theta_{p/q}=
 \frac{1}{\sqrt{q}}
 \sum_{n=0}^{q-1}
 e\Big(\frac{p n^2}{q}\Big)
 \qquad\text{verifies}\quad
 |\theta_{p/q}|^2=
 \begin{cases}
  2&\text{if }4\mid q,
  \\
  1&\text{if }2\nmid q,
  \\
  0&\text{otherwise}
 \end{cases}
\end{equation}
{from which}
it is not difficult to deduce that the absolute value of the partial sums of \eqref{riem1} tends to infinity  at $x=p/q$ with $p/q$ {irreducible} when $4\nmid q-2$. In particular the series diverges in a dense set. 
The following result gives a full characterization of the convergence points and shows that the divergence also occurs in irrational values extremely well approximated by rationals. 
Note that this sharpens \cite[Th.1.4]{SeUb}.

\begin{theorem}\label{convergence}
 Let $x\in [0,1)\setminus \mathbb Q$, with $\{p_j/q_j\}_{j=1}^{\infty}$ the convergents corresponding to its continued fraction. Then the Fourier series \eqref{riem1} converges if and only if
 \begin{equation}\label{app_series}
  \frac 12\sum_{j=1}^{\infty} \frac{\theta_{p_j/q_j}}{\sqrt{q_j}} \log \frac{q_{j+1}}{q_j}
 \end{equation}
 does, and in this case the difference between $F(x)$ and this sum is bounded by an absolute constant.
\end{theorem}

For instance, \eqref{riem1} diverges at  $x=\sum_{n=1}^\infty ({10\hspace{-1pt}\uparrow\uparrow\hspace{-1pt} n})^{-1}$ where we have used Knuth's up-arrow notation ${10\hspace{-1pt}\uparrow\uparrow\hspace{-1pt} 1}=10$ and ${10\hspace{-1pt}\uparrow\uparrow\hspace{-1pt} (n+1)}=10^{10\hspace{0pt}\uparrow\uparrow\hspace{0pt} n}$. Each partial sum gives a convergent $p_j/q_j$ with $q_j=10\hspace{-1pt}\uparrow\uparrow\hspace{-1pt} n$ \cite[Th.7.9.8]{MiTa} (although not every convergent comes from a partial sum) and $q_j=O(\log q_{j+1})$ \cite[(7.43)]{MiTa}. Hence the series \eqref{app_series} contains arbitrary large terms. 

\medskip

The previous result suggests that the series \eqref{riem1} is far for being bounded. The next one analyzes the level sets and reveals an intuitively different truth showing very small variations in small intervals.  

Here it is convenient to introduce the notation $I_{p/q}$
to mean the interval of numbers $x\in [0,1)$ such that their convergents include those of the irreducible fraction $p/q$. Equivalently, 
if 
$p/q=[0;a_1,a_2,\dots, a_{j_0}]$
then
$x=[0;a_1,a_2,\dots, a_{j_0},\dots]$.
There is a little ambiguity in this definition because for $a_{j_0}\ne 1$,
$p/q=[0;a_1,\dots, a_{j_0}-1,1]$ and hence the last but one convergent can be skipped. This uncertainty disappears once we fix one of the two possibilities for the length $j_0$. The following results hold irrespectively of this choice. 
We  also use farther the notation $A\asymp B$ to mean $c_1B\le A\le c_2 B$ with $c_1,c_2>0$ absolute constants.

\begin{theorem}\label{level_upper}
There exists a constant $C>0$ such that 
for any $I=I_{p/q}$ and $\lambda>0$
\[
 \big|
 \{x\in I\; :\; |F(x)-F_I|>\lambda\}
 \big|
 \le C
 |I|
 e^{-\lambda\sqrt{2q}}.
\]
\end{theorem}

Note that this is much stronger than the in general optimal inequality \eqref{joni_ineq} applied to our intervals. 

In fact this can be complemented with a lower bound.

\begin{theorem}\label{level_lower}
There exists a constant $C>0$ such that 
for any $I=I_{p/q}$ and $\lambda>0$
\[
 \big|
 \{x\in I\; :\; |F(x)-F_I|>\lambda\}
 \big|
 \ge C
 |I|
 e^{-c_q\lambda \sqrt{q}}
\]
with $c_q=\sqrt{2}$ if $4\mid q$,  $c_q=2$ if $2\nmid q$ and  $c_q=2\sqrt{2}$ otherwise. In particular, by Theorem~\ref{level_upper},
if $4\mid q$
\[
 \big|
 \{x\in I\; :\; |F(x)-F_I|>\lambda\}
 \big|
 \asymp
 |I|
 e^{-\lambda\sqrt{2q}}.
\]
\end{theorem}

\

The proof of these results is based on a relation between the series \eqref{app_series} and the oscillation that has independent interest.
According to Theorem~\ref{convergence}, \eqref{app_series} acts as a proxy for the function $F(x)$. On the other hand, $|I_{p/q}|$ is comparable to $q^{-2}$ and the intuitive ideas backing the uncertainty principle suggest that only the frequencies less that $q^2$ matter for the the average on~$I_{p/q}$. This leads to suspect that the part of \eqref{app_series} corresponding to $q_j\ge q$ is the one giving us information about the oscillation.  The next result makes this program rigorous in a quite precise form.

\begin{proposition}\label{F_size_on_I}
For $I=I_{p/q}$ and $x\in I$ with convergents $\{p_j/q_j\}_{j=1}^\infty$, we have 
\[
 F(x)-F_I 
 =
 \frac 12\sum_{q_j\ge q} \frac{\theta_{p_j/q_j}}{\sqrt{q_j}} \log \frac{q_{j+1}}{q_j}
 +
 O\Big(\frac{1}{\sqrt q}\Big)
\]
with an absolute $O$-constant. 
\end{proposition}

We finish this section pointing out that there is also an upper bound for the measure of the level sets of the oscillation in arbitrary intervals but it depends, so to speak, on the rationals with smaller denominator that the considered interval contains. This is the meaning of the next result because any subinterval of $[0,1)$ is contained in some~$I_{p/q}$.

\begin{proposition}\label{subintervals}
 The conclusion of Theorem~\ref{level_upper} still holds when $I$ is an interval included in $I_{p/q}$. 
\end{proposition}

{\sc Remark.}
Following the ideas in the proof one can actually show that for intervals $I\subset I_{p/q}$ with extremes 
$p/q=[0;a_1,\dots, a_{j_0}]$
and
$[0;a_1,\dots, a_{j_0}, c]$
the lower bound in Theorem~\ref{level_lower} holds.
Notice that its length goes to~$0$ when $c\to \infty$. This shows that the sizes of the level sets of any interval depend on the smallest interval $I_{p/q}$ containing $I$.

\section{Proof of the results}

\begin{proof}[Proof of Theorem~\ref{general}]
 Let $g(x)=\sum_{n=1}^\infty   b_ne(\nu_n x)$
 and $I\subset\mathbb{T}$ an interval with $|I|=\epsilon$. 
 Write $g=g_1+g_2$ with $g_1$ the part of the sum with $n\le N$ and $g_2$ the rest. 
 By the mean value theorem applied to the real and imaginary part of $g_1$ 
 \[
  \|g_1\|_I
  \le
  \Big\|g_1-\int_I g_1\Big\|_\infty
  \le
  4\pi\epsilon S_N.
 \]
 On the other hand, $\|g_2\|_I\le 2\epsilon^{-1}\int_I |g_2|$ and Jensen's inequality implies
 \[
  \frac{\epsilon}{4}
  \|g_2\|_I^2
  \le
  \int_I |g_2|^2
  =
  \epsilon
  \sum_{n=N+1}^\infty |b_n|^2
  +
  \sum_{n=N+1}^\infty
  \ 
  \sum^\infty_{\substack{ m=N+1\\ m\ne n } } 
  \frac{b_ne(\nu_nx) \overline{b_me(\nu_mx)}}{2\pi i(\nu_n-\nu_m)}
  \Big|_{x=r}^s
 \]
 where $I=[r,s]$. {Applying \eqref{hilbert}} to the double sum evaluated at $x=r$ and $x=s$, we have
 \[
  \|g_2\|_I^2
  \le
  4
  \sum_{n=N+1}^\infty |b_n|^2
  +
  \frac{6}{\epsilon}
  \sum_{n=N+1}^{\infty}
  \frac{|b_n|^2}{\min(\nu_n- \nu_{n-1}, \nu_{n+1}-\nu_n)}.
 \]
 Hence $\|g\|_I\le 4\pi\epsilon S_N+(6\epsilon^{-1} T_N + 4 \sum_{n=N+1}^\infty a_n^2)^{1/2}$ for every $N\in \Z^+$ {implying that} $\|g\|_*\le \kappa$.
\end{proof}

\begin{proof}[Proof of Proposition~\ref{norm_I}]
 Following the steps of the proof of Theorem~\ref{general} we have to bound $\kappa$ but with the supremum restricted to $\epsilon<\epsilon_0$ with $\epsilon_0$ arbitrarily small. For each $\epsilon$ we are going to choose $N$ such that 
 \[
  \nu_N\le 
  \epsilon^{-1}\Big(\frac{3}{16\pi^2}\Big)^{1/3}
  < \nu_{N+1}.
 \]
 If $\epsilon_0$ is small enough then $\sum_{N+1}^\infty a_n^2$ is as small as we want and then the result follows if we prove for this choice of $N$ {the following estimate:}
 \begin{equation}\label{red_prop}
  \limsup_{\epsilon\to 0^+}
  \Big(
  4\pi \epsilon S_N
  +\big(6\epsilon^{-1} T_N\big)^{1/2}
  \Big)
  \le 
  3(12\pi)^{1/3}\delta^{-1}.
 \end{equation}
 Under our hypothesis
 \[
  \limsup_{\epsilon\to 0^+}
  \big(
  \epsilon S_{N-1}
  \big)
  \le
  \limsup_{\epsilon\to 0^+}
  \Big(
  \epsilon\sum_{n=1}^{N-1}
  \big(\nu_{n+1}/\nu_n-1)\nu_n\delta^{-1}
  \Big).
 \]
 The sum telescopes and we get, since $S_N=S_{N-1}+o(\epsilon \nu_N\delta^{-1})$ (note that $a_N\to 0$),
 \begin{equation}\label{bSN}
  \limsup_{\epsilon\to 0^+}
  \big(
  4\pi \epsilon S_N
  \big)
  \le 
  \limsup_{\epsilon\to 0^+}
  \big(
  4\pi \epsilon \nu_N\delta^{-1}
  \big)
  \le
  (12\pi)^{1/3}\delta^{-1}.
 \end{equation}

 To bound $T_N$,  we claim {that}
 \[
  \frac{a_n^2}{M(n)}
  \le
  \delta^{-2}(1+\delta a_n)
  \Big|
  \frac{1}{\nu_{n}}-\frac{1}{\nu_{n+\eta}}
  \Big|
 \]
 where $\eta = \pm 1 $ with $M(n)= |\nu_n-\nu_{n+\eta}|$. Note that this claim is equivalent to 
 \[
  \frac{a_n^2\nu_n\nu_{n+\eta}}{\big(\nu_{n+\eta}-\nu_n\big)^2}
  \le
  \delta^{-2}(1+\delta a_n).
 \] 
 If $\eta=1$,  
 $\nu_{n+1}\ge (1+\delta a_n)\nu_n$  and using that $\nu_{n+1}/\big(\nu_{n+1}-\nu_n\big)^2$ decreases 
 in~$\nu_{n+1}$ we get
 \[
  \frac{a_n^2\nu_n\nu_{n+1}}{\big(\nu_{n+1}-\nu_n\big)^2}
  \le
  \frac{a_n^2\nu_n(1+\delta a_n)\nu_n}{\big((1+\delta a_n)\nu_n-\nu_n\big)^2}
  =
  \delta^{-2}(1+\delta a_n).
 \]
 If $\eta = -1$ the argument holds using  $\nu_{n}\ge (1+\delta a_n)\nu_{n-1}$.
 
 Taking into account the claim, we conclude
 \[
  \limsup_{\epsilon\to 0^+}
  \big(
  6\epsilon^{-1} T_N
  \big)
  \le 
  \limsup_{\epsilon\to 0^+}
  \Big(
  12\epsilon^{-1} \delta^{-2}
  \sum_{n=N+1}^\infty
  \big(
  \frac{1}{\nu_{n}}-\frac{1}{\nu_{n+1}}
  \big)
  \Big)
 \]
 where {a} $2$ factor comes from the two possibilities $\eta=\pm 1$. Since the sum telescopes to $\nu_{N+1}^{-1}$ we get
 \begin{equation}\label{bTN}
  \limsup_{\epsilon\to 0^+}
  \big(
  6\epsilon^{-1} T_N
  \big)
  \le 
  4 (12\pi)^{2/3}\delta^{-2}. 
 \end{equation}
 
 {The estimates} \eqref{bSN} and \eqref{bTN} allow us to complete the proof of \eqref{red_prop}.
\end{proof}

\

Theorem~\ref{convergence} is a consequence of the following {Proposition~\ref{partial},} because the partial sum of $F_N$ corresponding to $n<N$  can be written as 
$\sum_{j\le J} \big(F_{q_{j+1}}-F_{q_j}\big) - \big(F_{q_{J+1}}-F_{N}\big)$ with $q_{J}\le N<q_{J+1}$ and $F_{q_{j+1}}-F_{q_j}$ gives each term in \eqref{app_series}.

\begin{proposition}\label{partial}
 With the  notation of Theorem~\ref{convergence}, for $q_j\le m<q_{j+1}$ we have
 \[
  \sum_{m\le n<q_{j+1}}
  \frac{e(n^2x)}{n}
  =
  \frac{\theta_{p_j/q_j}}{2\sqrt{q_j}} \log^+ 
  \Big(
  \frac{q_{j+1}q_j}{m^2}
  \Big)
  +
  O\big(q_j^{-1/2}\big)
 \]
 with an absolute $O$-constant and where $\log^+ t=\max(\log t, 0)$.
\end{proposition}

\begin{proof}
 Write $x=p_j/q_j+h_j$. Then $1/2<|h_j|q_jq_{j+1}<1$ \cite[\S7.5]{MiTa}. 
 By \cite[Th.6]{FiJuKo} (this is implicit in the classic work by Hardy and Littlewood \cite{HaLi}) the contribution to the sum of the values with $Cq_{j+1}\le n< q_{j+1}$ is absorbed by the error term. Hence we can assume $m<M$ with $M=q_{j+1}/8$ and restrict the sum to $n<M$. 
 
 For $4\nmid q_j-2$ \cite[Th.5]{FiJuKo} with $\theta=\alpha=A=0$ gives 
 \begin{equation}\label{fiedler}
  \sum_{m\le n<N}
  e(n^2x)
  =
  \frac{\theta_{p_j/q_j}}{\sqrt{q_j}}
  \int_m^N
  e(h_j t^2)\; dt
  +
  O\big(q_j^{1/2}\big).
 \end{equation}
 The result extends to the case $4\mid q_j-2$, in which $\theta_{p_j/q_j}=0$, taking in \cite[Th.6]{FiJuKo} $\theta = 0$, $\alpha=-A=1/2$ and noting that the integral appearing there is $O(q_j)$ by the van der Corput lemma. 
 
 From \eqref{fiedler} and applying Abel's summation formula we deduce
 \[
  \sum_{m\le n<M}
  \frac{e(n^2x)}n
  =
  \frac{\theta_{p_j/q_j}}{2\sqrt{q_j}}
  \int_m^M
  2t^{-1} e(h_j t^2)\; dt
  +
  O\big(q_j^{-1/2}\big).
 \]
 The integral is {$\text{Ei}\big(2\pi i|h_j|M^2\big)-\text{Ei}\big(2\pi i|h_j|m^2\big)$ if $h_j>0$ and its conjugate if $h_j<0$,}
 with $\text{Ei}$ the exponential integral function (see \cite[8.233]{GrRy}). 
 Using $\text{Ei}(ix)=\min(0,\log x)+O(1)$ for $x>0$ \cite[8.215, 8.232]{GrRy} 
 and recalling $|h_j|q_jq_{j+1}\in (1/2,1)$ 
 the proof is complete.
\end{proof}

In the proof of Theorem~\ref{level_upper}, Theorem~\ref{level_lower} and Proposition~\ref{F_size_on_I} we are going to use the following auxiliary result from the metric theory of continued fractions.

\begin{lemma}\label{metric}
 For $x\in [0,1)\setminus\Q$ let $a_j(x)$ be the $j$-th partial quotient of $x$. Let $I=I_{p/q}$ and $j_0$ the length of the continued fraction of $p/q$ i.e., $p/q=p_{j_0}/q_{j_0}$. Then 
 \begin{equation}\label{prob_quot}
  \big|
  \{x\in I\;:\; a_j(x)=k\}
  \big|
  \asymp k^{-2}|I|
  \qquad
  \text{for any $j>j_0$ and $k\in \Z^+$}.
 \end{equation}
 Moreover, if $\{A_n\}_{n=1}^\infty\subset\R$ verifies $A_n\ge 1$ and $S=\sum A_n^{-1}<\infty$ then 
 \begin{equation}\label{prob_seq}
  \log |I| - \log \big|
  \{x\in I\;:\; a_{j_0+n}(x)\le A_n\text{ for }n\in\Z^+\}
  \big|
  \asymp S.
 \end{equation}
\end{lemma}

\begin{proof}
 The first formula is a particular case of \cite[Th.34]{khinchin}.
 
 For the second formula we adapt the proof of \cite[Th.10.2.4]{MiTa} or \cite[Th.30]{khinchin}. We can assume $A_n\in\Z^+$ because $\lfloor A_n\rfloor\asymp A_n$.
 Consider the nested sets
 \[
  E_0=I
  \qquad\text{and}\qquad
  E_N=
  \big\{
  x\in I\;:\;
  a_{j_0+n}(x)\le A_n
  \text{ for }1\le n\le N
  \big\}.
 \]
 Freezing the values of $a_{j_0+n}$ for $1\le n\le N-1$ to apply \eqref{prob_quot} and summing later on them in the range $[1,A_n]$ it follows that
 \[
  |E_{N-1}-E_N|
  \asymp
  |E_{N-1}|
  \sum_{k>A_N} k^{-2}
  \asymp
  A_N^{-1}
  |E_{N-1}|.
 \]
 Subtract $|E_{N-1}|$ to this expression and note on the other hand that \eqref{prob_quot} with $j=j_0+N$ and $k=1$ gives
 $|E_{N}|\ge c_1|E_{N-1}|$. Then for $N\ge 1$
 \[
  \max\big(
  c_1, 1-c_2A_N^{-1}
  \big)
  |E_{N-1}|
  \le
  |E_{N}|
  \le
  \big(
  1-c_3A_N^{-1}
  \big)
  |E_{N-1}|.
 \]
 Taking logarithms and using $-x/(1-c)\le\log(1-x)\le -x$ for $x\in [0,c)$ with $c<1$ we deduce
 \[
  \sum_{N=1}^\infty
  \log
  \frac{|E_{N-1}|}{|E_{N}|}
  \asymp S
 \]
 and the sum telescopes to the left hand side of \eqref{prob_seq}.
\end{proof}

We also separate an elementary lemma that will appear in the proof of the upper bound.

\begin{lemma}\label{el_sum}
 Given $\alpha_1,\dots, \alpha_d\in\R^+$  and $X>1$ let 
 $\mathcal{N}=\big\{\vec{n}\in(\Z^+)^d\;:\; n_1^{\alpha_1}n_2^{\alpha_2}\dots n_d^{\alpha_d}>X\big\}$. 
 If $\alpha$ is the maximum of $\alpha_1,\dots, \alpha_d\in\R^+$ and this maximum is reached only once then
 \[
  \sum_{\vec{n}\in\mathcal{N}}
  (n_1n_2\cdots n_d)^{-2}=O\big(X^{-1/\alpha}\big).
 \]
\end{lemma}

\begin{proof}
 Assume $d>1$. Let $\delta<1$ be the ratio between $\alpha$ and the biggest $\alpha_j$  less than $\alpha$. Then the sum is at most $\sum_{nm^\delta>Y} n^{-2}m^{-2}\tau_{d-1}(m)$ where $Y=X^{1/\alpha}$ and $\tau_{d-1}(m)$ is the number of representation of $m$ as a product of $d-1$ factors. It is known that $\tau_{d-1}(m)=O(m^\epsilon)$ for any $\epsilon>0$ \cite[\S18.1]{HaWr}. Therefore choosing $\epsilon =(1-\delta)/2$ shows that the sum above is $O(Y^{-1})$. 
\end{proof}

\begin{proof}[Proof of Proposition~\ref{F_size_on_I}]
 We perform a subdivision {similar to the one} in the proof of Theorem~\ref{general} but {observing} now {that} Proposition~\ref{partial} provides approximations rather than bounds. 
 
 Let $F^-$ be the partial sum of \eqref{riem1} corresponding to $n<q$ and let $F^+=F-F^-$. The mean value theorem applied to the real and imaginary part of~$F^-$ implies:
 \[
  \big| F^-(x)-F^-_I\big|
  \le 
  4\pi
  |I|
  \Big|
  \sum_{n<q} n e(n^2\xi)
  \Big|
  \qquad\text{for some }\xi\in I.
 \]
 The sum is $O\big(q^{3/2}\big)$ by \cite[Th.6]{FiJuKo} and $|I|<q^{-2}$. Hence
 \begin{equation}\label{Fmvt}
  F(x)-F_I
  =
  F^+(x)-F^+_I
  +O\big(q^{-1/2}\big)
 \end{equation}
 and Proposition~\ref{partial} with $m=q_j$ {yields:}
 \[
  F(x)-F_I
  =
  -F_I^+
  +
 \frac 12\sum_{q_j\ge q} \frac{\theta_{p_j/q_j}}{\sqrt{q_j}} \log \frac{q_{j+1}}{q_j}
 +
 O\big(q^{-1/2}\big).
 \]
 It only remains to prove that $F^+_I$ is negligible. For each $y\in I$ let $a_j(y)$ be the partial quotients in its continued fraction. The recurrence relation for the denominators of the convergents $p_j(y)/q_j(y)$ implies $q_{j+1}(y)/q_j(y)\le 2a_{j+1}(y)$. Then a new application of Proposition~\ref{partial} with $m=q_j$ gives
 \[
  F_I^+
  =
  { O\Big(
  q^{-1/2}
  +
  \sum_{q_j\ge q}
  \frac{1}{|I|}
  \int_I
  q_j^{-1/2}(y)
  \log a_{j+1}(y)\; dy
  \Big)  }
 \]
 {and this is $O\big(q^{-1/2}\big)$ because $q_j(y)$ grows as a geometric progression, in fact $q_{j+k-1}/q_j$ is at least the $k$-th Fibonacci number, and $(\log a_{j+1})_I=O(1)$ by \eqref{prob_quot} since $\sum k^{-2}\log k<\infty$.}
\end{proof}

\begin{proof}[Proof of Theorem~\ref{level_upper}]
 We can assume that $\lambda\sqrt{q}$ is larger than a constant because the result becomes trivial otherwise.
 
 Given $\vec{b}=(b_1,b_2,\dots, b_d)\in (\Z^+)^d$ with $d>2$ a constant to be fixed later{,} if 
 $p/q=[0;a_1,\dots, a_{j_0}]$
 let $I_{\vec{b}}$ be the interval $I_{p'/q'}$ corresponding to 
 $p'/q'=[0;a_1,\dots, a_{j_0},b_1,\dots, b_d]$. In this way $I=\bigcup I_{\vec{b}}$. By Proposition~\ref{F_size_on_I} and {the estimate} $q_{j+1}/q_j=a_{j+1} +O(1)$, we have for each $x\in I_{\vec{b}}$
 \[
  |F(x)-F_I|
  \le
  \sum_{k=1}^d
  \frac{\eta_k\log b_k}{\sqrt{q_{j_0+k-1}}}
  +
  \frac{1}{\sqrt{2}}
  \sum_{n=1}^\infty
  \frac{\log a_{j_0+d+n}(x)}{\sqrt{q_{j_0+d+n-1}(x)}}
  +O\big(q^{-1/2}\big)
 \]
 where $\eta_k=\frac 12 |\theta_{p_{j_0+k-1}/q_{j_0+k-1}}|$ which is at most $1/\sqrt{2}$ by \eqref{g_norm}.
 
 {Therefore the points} $x\in I_{\vec{b}}$ such that {$a_{j_0+d+n}(x)\le C n^2 e^{\lambda \sqrt{2q}}$ for $n\in\Z^+$ with $C$ a large enough constant} form a set of measure differing from $|I_{\vec{b}}|$ in {$O\big(|I_{\vec{b}}|e^{-\lambda \sqrt{2q}}\big)$}
 by \eqref{prob_seq}. {We can then} assume this bound and {observe also that} the second sum {is} $O\big(\lambda\sqrt{q/q_{j_0+d}}\big)$ {because $q_{j+2}/q_j>2$}. In this way we obtain an upper bound for $|F(x)-F_I|$ depending on $\vec{b}$ {but} not {on} $x$.
 Successive applications of \eqref{prob_seq} give
 \begin{equation}\label{up_sum}
  \big|
  \{x\in I\;:\;
  |F(x)-F_I|>\lambda\}
  \big|
  \le 
  |I|
  \sum_{\vec{b}\in\mathcal{B}_\lambda}
  (b_1b_2\cdots b_d)^{-2}
 \end{equation}
 with $\mathcal{B}_\lambda$ the set of $(b_1,\dots, b_d)$ such that
 \begin{equation}\label{cond_b}
  \sum_{k=1}^d
  \frac{\eta_k\log b_k}{\sqrt{q_{j_0+k-1}}}
  +
  C\lambda
  \sqrt{\frac{q}{q_{j_0+d}}}
  +Cq^{-1/2}>\lambda
 \end{equation}
 {where $C$ is a certain universal} constant. Choosing $d$ such that {$C^2q/q_{j_0+d}<1/5$} the sum contributes at least $\lambda/2$ (recall that we can assume that $\lambda\sqrt{q}$ is large), so that $b_k>e^{\lambda\sqrt{q}/(2d\eta_k)}$ for some $k$.
 Using the recurrence formulas
 $q_{j_0+d}/q\ge b_k$ and we {can conclude} that the second term in \eqref{cond_b} is $O\big(q^{-1/2}\big)$. Moreover {$q/q_{j_0+k-1}<1/2$ for $k>2$ and $q_{j_0+1}>q$}. {Therefore} for some constant $C$
 \[
  \mathcal{B}_\lambda
  \subset
  \Big\{
  \vec{b}\;:\;
  \eta_1 \log b_1
  +
  \eta_2\log b_2
  +
  \frac{1}{\sqrt{8}}
  \sum_{k=4}^d\log b_k
  >
  \lambda\sqrt{q}-C
  \Big\}
 \]
 
 Note that $\eta_1$ and $\eta_2$ cannot be simultaneously {equal to} $1/\sqrt{2}$ because {that} would require $4\mid q$ and $4\mid q_{j_0+1}$ and they are coprime. Hence $\mathcal{B}_\lambda$ is contained in a set $\mathcal{N}$ as in Lemma~\ref{el_sum} with $\alpha=1/\sqrt{2}$ and $X=C e^{\lambda\sqrt{q}}$ and the expected bound follows from \eqref{up_sum}. 
\end{proof}

\begin{proof}[Proof of Theorem~\ref{level_lower}]
 If $4\nmid q$ and $2\mid q$, the fraction $p'/q'$ obtained adding a last partial quotient~$1$ to the continued fraction of $p/q$ verifies $2\nmid q'$ and $q'<2q$. 
 We have $I_{p'/q'}\subset I_{p/q}$ and $|I_{p'/q'}|\asymp |I_{p/q}|$ hence the result in this case is deduced from the case $2\nmid q$.
 
 If $4\mid q$ or $2\nmid q$, by \eqref{g_norm}, Proposition~\ref{F_size_on_I} and $q_{j+1}/q_j=a_{j+1} +O(1)$, we have
 \[
  |F(x)-F_I|
  =
  \frac{\log a_{j_0+1}(x)}{c_q\sqrt{q}}
  +
  O\Big(
  q^{-1/2}
  +
  \sum_{n=1}^\infty
  \frac{\log a_{j_0+n+1}(x)}{\sqrt{q_{j_0+n}(x)}}
  \Big).
 \]
 
 Clearly $a_{j_0+1}(x)=k$ if and only if $x\in I_{p'/q'}$ with  $p'/q'$ obtained from $p/q$ adding a last partial quotient $k$ and therefore $q'\asymp kq$. Choosing for instance $A_n=n^2$ in \eqref{prob_seq} applied to $I_{p'/q'}$ we deduce the existence of $J_k\subset I_{p'/q'}$ with $|J_k|\asymp |I_{p'/q'}|$ such that
 \[
  |F(x)-F_I|
  =
  c_q^{-1}q^{-1/2}
  \log k
  +
  O\big(
  q^{-1/2}
  \big)
  \qquad\text{for}\quad x\in J_k.
 \]
 Hence if $\mathcal{K}=\{k\;:\; \log k > c_q\lambda\sqrt{q}+C\}$ with a large enough $C$
 \[
 \big|
 \{x\in I\; :\; |F(x)-F_I|>\lambda\}
 \big| 
 \ge
 \sum_{k\in\mathcal{K}}
 |J_k|
 \]
 and this yields the result because $|J_k|\asymp (q')^{-2}\asymp k^{-2}|I|$.
\end{proof}

\begin{proof}[Proof of Proposition~\ref{subintervals}]
 We start with some preliminary reductions. 
 Clear\-ly we can suppose that $\lambda\sqrt{q}$ is larger than a fixed constant.
 Say that the extremes 
 of the interval $I=(\alpha,\beta)$ are 
 \[
  \alpha=[0;a_1,\dots,a_{j_0},\alpha'] 
  \quad\text{and}\quad
  \beta=[0;a_1,\dots,a_{j_0}, \beta']
 \]
 with $p/q=[0;a_1,\dots a_{j_0}]$, 
 $\alpha' = [c;c_1,c_2,\dots]$
 and
 $\beta' = [d;d_1,d_2,\dots]$. 
 Let us suppose $j_0$ odd (the other case is completely similar) then $\alpha$ and $\beta$ are increasing in $c$ and $d$ and  we have $c\le d$. In fact we can assume $c<d$ because if $c=d$ we would take 
 $a_{j_0+1}=c$ to get  a bigger $q$ leading to a stronger result. 
 
 By the properties of the continued fractions we have $cd q^2 |I| \asymp \beta'-\alpha'$.
 But an issue appears here related by the slight ambiguity of continued fractions: that is 
 we have $\beta'-\alpha'\asymp d-c$ 
 except perhaps in the case $d=c+1$, $c_1=1$ because of the identity $[c;1]=[c+1]$. Let us assume now $c_1\ne 1$ whenever $d=c+1$ (we will discuss the remaining special case later).
 Under this assumption we have
 \begin{equation}\label{length_I}
  |I| \asymp q^{-2}\big(c^{-1}-d^{-1}\big)
 \end{equation}
 and it is enough to prove the result when $c_j=d_j=\infty$ i.e., when these partial quotients do not appear, because without these partial quotients and replacing $d$ by  $d+1$ we get a subinterval of $I_{p/q}$ that contains  to $I$ and has comparable measure by \eqref{length_I}.
 
 Let $I^b=I_{p'/q'}$ with $p'/q'=[0;a_1,\dots, a_{j_0},b]$. 
 Then $\bigcup_{b=c}^{d-1} I^b$
 differs from $I$ in a finite number of points and the triangle inequality assures
 \[
  \big|
  \{x\in I\; :\; |F(x)-F_I|>\lambda\}
  \big|
  \le 
  \sum_{b=c}^{d-1}
  \big|
  \big\{x\in I^b\; :\; |F(x)-F_{I^b}|>\lambda- |F_{I^b}-F_I|\big\}
  \big|.
 \]
 Combining the trivial estimate and Theorem~\ref{level_upper}, we have
 \begin{equation}\label{st_pr}
  \big|
  \{x\in I\; :\; |F(x)-F_I|>\lambda\}
  \big|
  =O
  \Big(
  \sum_{b=c}^{d-1}
  |I^b|\min\big(
  1,  
  e^{(|F_{I^b}-F_I|-\lambda)\sqrt{2bq}}
  \big)
  \Big).
 \end{equation}
 Our aim is to approximate $|F_{I^b}-F_I|$ and substitute it in this formula.
 
 \medskip
 
 Note that we obtained \eqref{Fmvt} from the mean value theorem on $I_{p/q}$ and that it also applies to any of its subintervals. Hence for $x\in I^b$ in the convergence set of $F$ we have
 \[
  F_{I^b}-F_I
  =
  \big(F(x)-F_I\big)
  +
  \big(F(x)-F_{I^b}\big)
  =
  F_{I^b}^+-F_I^+
  +O\big(q^{-1/2}\big)
 \]
 and Proposition~\ref{partial} approximates $F_{I^b}^+$ and $F_I^+$ giving
 \[
  F_{I^b}-F_I
  =
  \frac{\theta_{p/q}}{2\sqrt{q}}
  \big(\log b-(\log a_{j_0+1}(x))_I\big)
  +O(E)
  +O\big(q^{-1/2}\big)
 \]
 with
 \[
  E
  =
  \sum_{q_j>q}
  \bigg(
  \frac{1}{|I|}
  \int_{I}
  q_j^{-1/2}(y)\log a_{j+1}(y)\; dy
  +
  \frac{1}{|I^b|}
  \int_{I^b}
  q_j^{-1/2}(y)\log a_{j+1}(y)\; dy
  \bigg).
 \]
 Using \eqref{prob_quot} applied to each interval $I^b$ and the exponential growth of $q_j$, we have $E=O\big(q^{-1/2}\big)$. But this argument also proves
 $\big(\log (a_{j_0+1}(x)/c)\big)_I = O(1)$, hence
 $(\log a_{j_0+1}(x))_I=\log c+O(1)$ and we can conclude that
 \[
  F_{I^b}-F_I
  =
  \frac{\theta_{p/q}}{2\sqrt{q}}
  \log \frac{b}{c}
  +O\big(q^{-1/2}\big).
 \]
 By \eqref{g_norm}, $\sqrt{2q}|F_{I^b}-F_I|\le \log(b/c)+O(1)$
 and the minimum in \eqref{st_pr} is reached by the exponential for $b$ in 
 \[
  \mathcal{B}=\big\{
  b\in [c, d-1]\cap \Z\;:\; \log (b/c)
  \le \lambda\sqrt{2q}-C
  \big\}
  \quad
  \text{with $C$ a constant}.
 \]
 Note that by our initial assumption on $\lambda$, we can suppose $\lambda\sqrt{2q}-C$ to be greater than a large positive constant. In particular, if $d\le 2c$ we can suppose that $\mathcal{B}$ includes the whole range. In this case, the exponential decay and $|I^b|\asymp b^{-2}q^{-2}$ shows that the contribution is comparable to that of the first term $b=c$ and we get
 \[
  \big|
  \{x\in I\; :\; |F(x)-F_I|>\lambda\}
  \big|
  =O
  \big(
  q^{-2}c^{-2}e^{-\lambda\sqrt{2cq}}
  \big)
  =O
  \big(
  |I|e^{-\lambda\sqrt{2q}}
  \big).
 \]
 If $d>2c$ then $\mathcal{B}$ may not cover the whole range and we have to add the contribution in \eqref{st_pr} of the remaining terms in which the minimum may be~$1$, namely
 \begin{equation}\label{notinB}
  \sum_{b\not\in\mathcal{B}}
  |I^b|
  =
  O\Big(
  q^{-2}
  \sum_{b\not\in\mathcal{B}}
  b^{-2}
  \Big)
  =
  O\big(
  q^{-2}
  \max_{b\not\in\mathcal{B}}
  b^{-1}
  \big)
  =
  O\big(
  q^{-2}
  c^{-1}
  e^{-\lambda\sqrt{2q}}
  \big)
 \end{equation}
 and this proves the result because under $d>2c$ \eqref{length_I} implies $q^{-2}c^{-1}\asymp |I|$. 

 \medskip 
 
 Finally, we mention how to deal with the special case $d=c+1$, $c_1=1$. In this case we redefine $p/q$
 as
 $
 [0;a_1;\dots, a_{j_0},c,1]
 =
 [0;a_1;\dots, a_{j_0},c+1]
 $
 and
 consider $I_1^b$ and $I_2^b$ as the intervals $I^b$ for each of these representations of $p/q$.
 Proceeding as before, we can assume that the extremes of the interval $I$ are 
 $
 [0;a_1;\dots, a_{j_0},c,1,c_2]
 $
 and
 $
 [0;a_1;\dots, a_{j_0},c+1,d_1]
 $.
 It is easy to see that 
 $
  \bigcup_{b=c_2}^\infty
  I_1^b
  \cup
  \bigcup_{b=d_1}^\infty
  I_2^b
 $
 differs from $I$ in a countable set and $|I|\asymp q^{-2}\max(c_2^{-1},d_1^{-1})$ by \eqref{length_I}
 applied to each of these unions.
 We can repeat the proof above, with a bigger $q$, replacing formally $I^b$ by $I_j^b$ and $c$ by $c_2$ or $d_1$ and $d$ by $\infty$ to get in $\mathcal{B}$
 $q^{-2}\max(c_2^{-1}, d_1^{-1})$ instead of $q^{-2}c^{-1}$, which matches the measure of $I$. 
\end{proof}

\bibliography{bibbmo}
\bibliographystyle{plain}

\end{document}